\documentclass[12pt]{article}

\usepackage{cmap}  
\usepackage[T2A]{fontenc}
\usepackage[utf8]{inputenc}
\usepackage[russian, english]{babel}
\usepackage[usenames]{color}
\usepackage{amsmath,amssymb,amsthm}
\usepackage{arcs}
\usepackage{epigraph}
\usepackage{epsfig}
\usepackage{graphicx}
\usepackage[pdftex,colorlinks=true,linkcolor=blue,urlcolor=red,unicode=true,hyperfootnotes=false,bookmarksnumbered]{hyperref}
\usepackage{ifthen}
\usepackage{import}
\usepackage{indentfirst}
\usepackage{mathtools}
\usepackage{cancel}
\usepackage{xcolor}

\newcommand{\eps}{\varepsilon}

\renewcommand{\le}{\leqslant}
\renewcommand{\ge}{\geqslant}

\newcommand{\ff}{\ensuremath{\mathcal{F}}}

\newcommand{\E}{\mathsf{E}}
\newcommand{\D}{\mathsf{Var}}
\newcommand{\Prb}{\mathsf{P}}

\newtheorem{thm}{Theorem}
\newtheorem{prob}[thm]{Problem}
\newtheorem{lemma}[thm]{Lemma}

\newtheorem{prop}[thm]{Proposition}
\newtheorem{conj}{Conjecture}

\title{Sharp bounds for the chromatic number of random Kneser graphs}
\author{Sergei Kiselev\footnote{MIPT, Moscow; Email: {\tt kiselev.sg@gmail.com}}, Andrey Kupavskii\footnote{IAS Princeton; CNRS, France; MIPT, Moscow; Email: {\tt kupavskii@yandex.ru} \ \ The authors acknowledge the financial support from the Russian Government in the framework of MegaGrant no 075-15-2019-1926.}}
\date{}
\begin{document}
\maketitle

\begin{abstract} Given positive integers $n\ge 2k$, the {\it Kneser graph} $KG_{n,k}$ is a graph whose vertex set is the collection of all $k$-element subsets of the set $\{1,\ldots, n\}$, with edges connecting pairs of disjoint sets. One of the classical results in combinatorics, conjectured by Kneser and proved by Lov\'asz, states that the chromatic number of $KG_{n,k}$ is equal to $n-2k+2$. In this paper, we study the chromatic number of the {\it random Kneser graph} $KG_{n,k}(p)$, that is, the graph obtained from $KG_{n,k}$ by including each of the edges of $KG_{n,k}$ independently and with probability $p$.

We prove that, for any fixed $k\ge 3$, $\chi(KG_{n,k}(1/2)) = n-\Theta(\sqrt[2k-2]{\log_2 n})$, as well as $\chi(KG_{n,2}(1/2)) = n-\Theta(\sqrt[2]{\log_2 n \cdot \log_2\log_2 n})$. We also prove that, 
for 
$k\ge (1+\varepsilon) \log\log n$, we have $\chi(KG_{n,k}(1/2))\ge n-2k-10$. This significantly improves previous results on the subject, obtained by Kupavskii and by Alishahi and Hajiabolhassan. The bound on $k$ in the second result is also tight up to a constant. We also discuss an interesting connection to an extremal problem on embeddability of complexes.

\end{abstract}

\section{Introduction}

For positive integers $n,k$, where $n\ge 2k$, the {\it Kneser graph} $KG_{n,k}=(V,E)$ is the graph, whose vertex set $V$ is the collection of all $k$-element subsets of the set $[n]:=\{1,\ldots, n\}$, and $E$ is the collection of the pairs of disjoint sets from $V$. This notion was introduced by Kneser \cite{Knes}, who showed that $\chi(KG_{n,k})\le n-2k+2$. He conjectured that, in fact, equality holds in this inequality. This was proved by Lov\'asz \cite{Lova}, who  introduced the use of topological methods in combinatorics in that paper. One of the motivations for the studies in the present paper is to develop tools that could potentially help to obtain a `robust' combinatorial proof of Kneser's conjecture.

We remark that independent sets in $KG_{n,k}$ are {\it intersecting families}, and it is a famous result of Erd\H os, Ko and Rado \cite{EKR} that $\alpha(KG_{n,k}) = {n-1\choose k-1}$.

The notion of the {\it random Kneser graph} $KG_{n,k}(p)$ was introduced in \cite{BGPR,BGPR2}. For $0<p< 1$, the graph $KG_{n,k}(p)$ is constructed by including each edge of $KG_{n,k}$ in $KG_{n,k}(p)$ independently with probability $p$. 
The authors of \cite{BNR} studied the independence number of $KG_{n,k}(p)$. Later, their results were strengthened in \cite{BBN,DT,DK}). Interestingly, the independence number of $KG_{n,k}(p)$ stays {\it exactly} the same as the independence number of $KG_{n,k}$ in many regimes. Independence numbers of random subgraphs of generalized Kneser graphs and related questions were studied in \cite{BKR, BGPR, BGPR2, Pyad, PR}. In \cite{Kup}, the second author  proposed to study the chromatic number of $KG_{n,k}(p)$. 
He proved that in various regimes the chromatic number of the random Kneser graph is very close to that of the Kneser graph. In particular, he showed that for any constant $k$ and $p$ there exists a constant $C$, such that a.a.s. (asymptotically almost surely)\footnote{with probability tending to $1$ as $n\to \infty$.}
\begin{equation}\label{eq01}\chi(KG_{n,k}(p))\ge n- C n^{\frac 3{2k}}.\end{equation}
Moreover, he showed that the same holds for the random Schrijver graph (defined analogously based on Schrijver graphs, cf. \cite{Kup}). A better a.a.s. bound $\chi(KG_{n,k}(p))\ge n- C n^{\frac {2+o(1)}{2k-1}}$ was next obtained by Alishahi and Hajiabolhassan \cite{AH}. In a follow-up paper, the second author \cite{Kup2} improved the inequality \eqref{eq01} to
\begin{equation}\label{eq02} \chi(KG_{n,k}(p))\ge n- C(n\log n)^{1/k}\end{equation}
for some $C=C(p,k)$.
The main result of this paper is the following theorem, which, in particular, significantly improves upon the bounds \eqref{eq01} and \eqref{eq02} and settles the problem in the case of constant $k$.
\begin{thm}\label{thm1}
    For any fixed $k\ge 3$ and $n\to\infty$, we a.a.s.  have
    $$
    \chi(KG_{n,k}(1/2)) = n - \Theta\big(\sqrt[2k - 2]{\log_2 n}\big).
    $$
    
    For $k=2$ and $n\to\infty$ we a.a.s. have
    $$
    \chi(KG_{n,k}(1/2)) = n - \Theta\big(\sqrt[2]{\log_2 n \cdot \log_2\log_2 n}\big).
    $$
\end{thm}
For clarity, all our results are stated and proved for $p=1/2$. However, they are easy to extend to any constant or slowly decreasing $p$.
For $k=1$, $KG_{n,k}$ is just the complete graph $K_n$, and thus $KG_{n,1}(p) = G(n,p)$. Therefore, we a.a.s. have $\chi(KG_{n,1}(p))=\Theta(\frac n{\log n})$ (see, e.g., \cite{AS}), that is, an analogue of Theorem~\ref{thm1} cannot hold. We note that a weaker version of Theorem~\ref{thm1} was announced in the short note due to the first author and Raigorodskii \cite{KR}.

The papers \cite{Kup}, \cite{AH}, \cite{Kup2} were also concerned with the following question: when does the chromatic number drop by at most an additive constant factor? The best results here are due to the second author \cite{Kup2}, who proved the following a.a.s. bound for any fixed $l\ge 2$ and some absolute constant $C=C(l)$:
\begin{equation}\label{eqgrowk}\chi(KG_{n,k}(1/2))\ge n-2k+2-2l \ \ \ \ \textit{if }\ \ \ k \ge C(n\log n)^{1/l}.
\end{equation}

In this paper, we provide a major improvement of \eqref{eqgrowk}, replacing the polynomial dependence of $k$ on $n$ by doubly logarithmic.
\begin{thm}\label{thm3} For any $\eps>0$ and $n\to \infty$ we a.a.s. have
\begin{equation}\label{eqgrowknew}\chi(KG_{n,k}(1/2))\ge n-2k-10 \ \ \ \ \textit{if }\ \ \ k  \ge (1+\eps)\log_2\log_2 n.
\end{equation}
\end{thm}
The bound on $k$ is tight up to a constant: in  \cite[Section~5]{Kup2} it is proved that the bound $\chi(KG_{n,k}(1/2))\ge n-2k+2-2l$ cannot hold for any fixed $l$ if $k<(\frac 12 -\eps)\log_2 \log_2 n.$\\

\textbf{Remark. } In an earlier version of this paper, we showed the bound $\chi(KG_{n,k}(1/2))\ge n-2k+2-2l$ for $k  \ge C\log^{\frac 1{2l-3}} n.$ We also related this problem  to certain extremal properties of complexes. See Section~\ref{sec21} for the statement of the problem concerning simplicial complexes and the discussion section for the explanation of the relationship between the two problems. Since then, Kaiser and Stehl\'{\i}k provided an important construction of a sparse subgraph of Kneser graphs, which lead to the present improvement. We discuss their construction in Section~\ref{secxg}. We also note that the best possible constant in front of the double log is directly related to the value of $\zeta$ from Problem~\ref{pr2}.\vskip+0.1cm





Generalizing the notion of a Kneser graph, Alon, Frankl and Lov\'asz in \cite{AFL} studied the {\it Kneser hypergraph} $KG_{n,k}^r$. The vertex set of the Kneser hypergraph is the same as that of the Kneser graph $KG_{n,k}$, and the set of edges is formed by the $r$-tuples of pairwise disjoint sets. In particular, verifying a conjecture of Erd\H os \cite{Erd1}, they determined that $\chi(KG_{n,k}^r) = \big\lceil\frac{n-r(k-1)}{r-1}\big\rceil$. Determining the independence number of $KG_{n,k}^r$ is a much harder problem due to Erd\H os \cite{E} (known under the name of the Erd\H os Matching Conjecture), which remains unresolved in full generality. The best known results in this direction were obtained in \cite{F4} and, more recently, in \cite{FK13}. See also \cite{EKL, FK6, FK7} for related stability results.

We can define the {\it random Kneser hypergraph} $KG_{n,k}^r(p)$ in a similar way. Studying the chromatic number of $KG_{n,k}^r(p)$ was proposed by the second author in \cite{Kup}. First lower bounds were obtained in \cite{AH} and then they were significantly improved in \cite{Kup2}. In particular, it was shown that for fixed $k$, $p$ and $r\ge 3$ there exists a constant $C$, such that a.a.s.
\begin{equation}\label{eq03}\chi(KG_{n,k}^r(p))\ge \frac {n}{r-1}- C \log^{\frac 1{\alpha}} n,\end{equation}
where $\alpha = r(k-1)-\frac {2r-1}{r-1}$.
Note that $\chi(KG_{n,k}^r) = \frac{n}{r-1} - O(1)$ for fixed $k,r$. In the third theorem, we give an almost matching upper bound for \eqref{eq03}.

\begin{thm}\label{thm2}
    Let $k, r$ be fixed. Then there exists a constant $C$, such that for $n\to\infty$ a.a.s. we have
    $$
    \chi(KG^r_{n,k}(1/2)) \le \frac{n}{r-1}- C \log^{\frac{1}{r(k-1)}} n.
    $$
\end{thm}

Theorem~\ref{thm2} implies the upper bound from Theorem~\ref{thm1} for $k\ge 3$ and will be proved in Section~\ref{sec4}. The proof of the upper bound from Theorem~\ref{thm1} for $k=2$ is given after that, in Section~\ref{sec41}.

We do not go into more historical details here, and refer the reader to \cite{Kup2} for a longer introduction to the subject and a more detailed comparison of the bounds in different regimes.

We note that, while the methods for studying Kneser graphs and hypergraphs in \cite{Kup}, \cite{AH} were topological, in \cite{Kup2} combinatorial methods relating the structure of $KG_{n,k}$ and $KG_{n,k+l}$ were used. In this paper we use (different) combinatorial and probabilistic methods, which are based on the analysis of the structure of independent sets in $KG_{n,k}(p)$. The proof of Theorem~\ref{thm3} both in the new and old versions uses only one topological statement as a black box: the fact that chromatic number of a certain graph ($KG_{n,k}$ or its subgrpaph $XG_{n,k}$, to be defined later) is $n-2k+2$.

In the next section we discuss some problems related to colorings of Kneser graphs; in Section~\ref{sec3} we give the proofs of the lower bound for Theorem~\ref{thm1} and of Theorem~\ref{thm3}. In Section~\ref{sec4} we prove Theorem~\ref{thm2}. In Section~\ref{sec5} we conclude, state  some open problems and discuss the relationship between some of the questions from Section~\ref{sec21} to colorings of random Kneser graphs.

\section{Kneser colorings}\label{sec2}

One of the difficulties that arises in the study of Kneser graphs is that we poorly understand the structure of a union of several intersecting families. 

\begin{prob}\label{pr1} Given $n$, $k$ and $t$, where $n\ge 2k+t-1$, what is the maximum size of a union of $t$ intersecting families $\ff_1,\ldots, \ff_t \subset {[n]\choose k}$? What can one say about the structure of such families?
\end{prob}
Some of the results in this direction come from the Erd\H os Matching Conjecture: under a weaker assumption that the family does not contain $t+1$ pairwise disjoint sets, and $n$ is at least $(2t+1)k-t$, Frankl~\cite{F4} showed that the largest such family must consist of all the sets intersecting a subset of size $t$. In a recent work, Frankl and the second author \cite{FK13} resolved the Erd\H os Matching Conjecture for $n\ge \frac 53 tk-\frac 23 t$ and $t\ge t_0$. The extremal family is a union of $t$ intersecting families, and thus we get the same result for the problem above. The problem was also addressed in \cite{FF84} and \cite{EL}, where the case of constant $t$ was studied. However, the more interesting cases of Problem~\ref{pr1} are when $t$ is of order $n$, and thus $n<kt$. In this range the Erd\H os Matching Conjecture becomes trivial, and we have very little understanding of how this maximum should behave.

In the same vein, since there is no combinatorial proof of Kneser's conjecture\footnote{We note that, nevertheless, there are some truly combinatorial arguments that give rather sharp lower bounds.}, even our understanding of the possible structure of optimal colorings of Kneser graphs is very limited. (We should note here that there is a paper by Matou\v sek with the title ``A combinatorial proof of Kneser's conjecture'' \cite{MatKne}. Still, the argument there is a ``combinatorialization'' of the known topological proof and, as such, does not give new insights into the problem.)

We know that, provided $n$ is large in comparison to $k$,  in every proper coloring there must be a color that forms a {\it star}: an intersecting family in which all sets contain a given element. Their presence allows us to reduce the study of  colorings of $KG_{n,k}$ to that of $KG_{n-1,k}$ by simply removing the center of the star and the corresponding color. This is the observation that motivates Lemma~\ref{lem2} and plays the key role in the proofs of Theorems~\ref{thm1} and~\ref{thm3}.  The existence of such `star' colors is easy to show for $n>k^3$: it simply follows from the Hilton--Milner theorem \cite{HM} and the pigeon-hole principle, which implies that there must be a color class larger than the size of a Hilton--Milner family.

Thus, it is natural to ask the following question:
\begin{prob}\label{pr66} Given $k$, what is the largest number $n(k)$, such that there exists a proper coloring of $KG_{n,k}$ with $n-2k+2$ colors (or, a covering of ${[n]\choose k}$ by intersecting families) without any color (family) being a star?\end{prob}
Initially, our intuition was that, unless $n$ is close to $2k$, all colorings of $KG_{n,k}$ should be somewhat resemblant of the `canonical coloring': color all sets containing $i$ into the $i$-th color, until left with subsets of $[n-2k+2,n]$, which you color in the same color. However, this intuition was wrong.

\begin{prop}
  We have $n(k)\ge 2(k-1)^2$ for $k\ge 3$.
\end{prop}
\begin{proof} Put $n:=2(k-1)^2$ and split the ground set $[n]$ into $k-1$ blocks $A_i$ of size $2k-2$. For each block, say, $A_1:=[2k-2]$, consider the following covering by intersecting families: for $i=1,\ldots, 2k-5$, define
$$F_i:=\{2k-4,2k-3,2k-2\}\cup \{i+1,\ldots, i+k-3\},$$
where the addition and substraction in the second part of the set is modulo $2k-5$ (and thus the elements belong to $[2k-5]$). Consider the intersecting Hilton-Milner-type families of the form
$$\mathcal H_i:=\Big\{F\in {[n]\choose k}: i\in F, F\cap F_{i}\ne\emptyset\Big\}\cup F_{i}.$$
Complement it with the intersecting family
$$\mathcal G:=\Big\{F\in {[n]\choose k}: |F\cap \{2k-4,2k-3,2k-2\}|\ge 2\Big\}.$$
If a set $G\cap [2k-2]\supset \{i,j\}$ for $i<j$, then $G$ is contained in one of the families $\mathcal G$, $\mathcal H_l,$ $l\in [2k-5]$. Indeed, \begin{itemize}
\item if $\{i,j\}\subset \{2k-4,2k-3,2k-2\}$, then $G\subset \mathcal G$;
\item if $i<2k-4\le j$, then $G\subset \mathcal H_i$;
\item if $j<2k-4$ and $j\le i+k-3$, then $G\subset \mathcal H_i$; 
\item if $i+k-2\le j<2k-4$, then $j+k-3$ mod $2k-5$ is at least $i$, and $G$ is contained in $\mathcal H_j$.
\end{itemize}
Therefore, any set intersecting $A_1$ in at least $2$ elements is contained in one of the intersecting families given above. On the other hand, any $k$-set must intersect one of the $k-1$ blocks in at least $2$ elements. Thus, considering similar collections of intersecting families in the other blocks, we get that the whole of ${[n]\choose k}$ is covered.

We have $2k-4$ intersecting families on each block, which gives $(2k-4)(k-1)$ families in total. On the other hand, $\chi(KG_{n,k}) = 2(k-1)^2-2k+2 = 2(k-2)(k-1)$, that is, the number of intersecting families we used equals  the chromatic number of the graph. It is also clear that none of the families is a star, and we can easily preserve this property when making a coloring (rather than a covering).
\end{proof}
It is easy to see that, by slightly modifying the construction above, we can produce such a coloring for each $n\le 2(k-1)^2$.
We conjecture that such a coloring is impossible for bigger $n(k)$, and thus $n(k)= 2(k-1)^2$ for all $k\ge 3$. \\

\textbf{Remark.} In a recent paper \cite{KKKC}, we managed to show that $n(k) = (2+o(1))k^2$.

\subsection{Number of monochromatic edges and vertices in colorings of $KG_{n,k}$} \label{sec21} If we color $KG_{n,k}$ by $\chi(KG_{n,k})-1=n-2k+1$ colors, then we of course get at least $1$ monochromatic edge. But the intuition suggests that we should have much more. Modifying the standard coloring of $KG_{n,k}$ by coloring subsets on the last $2k$ elements (instead of $2k-1$) in the same color, we get a coloring in $n-2k+1$ colors with $\frac12 {2k\choose k}$ monochromatic edges.\footnote{It was pointed to us by an anonymous referee that Problem~\ref{pr2} and in particular this construction  is directly related to \cite[Section 3.5, Exercise~3 (due to Anders Bj\"orner)]{Mat}. This exercise asks for a construction of a coloring in $n-2k+1$ colors with $\frac12\binom{2k}{k}$ monochromatic edges and asks to show that this bound is best possible when either $k=2$ or $n=2k+1$.}

\begin{prob}\label{pr2} Given $k$ and $n$, what is the minimum number $\zeta$ of monochromatic edges in a coloring of $KG_{n,k}$ into $n-2k+1$ colors?
\end{prob}


One approach that allows to get some bounds on $\zeta$ is via {\it Schrijver graphs} $SG_{n,k}$, that is, induced subgraphs of Kneser graphs on $k$-sets not containing two cyclically consecutive elements of $[n]$. We call such $k$-sets {\it stable}. It is known that $\chi(SG_{n,k})=\chi(KG_{n,k})$. However, the number of vertices in $SG_{n,k}$ is roughly ${n-k\choose k}$. Thus, taking a $(n-2k+1)$-coloring of $KG_{n,k}$, we get monochromatic edges in each induced $SG_{n,k}$ (where different copies of $SG_{n,k}$ correspond to different orderings of $[n]$), and, averaging over the choice of the ordering, we can conclude that there are at least $\frac{|E(KG_{n,k})|}{|E(SG_{n,k})|}$ monochromatic edges in $KG_{n,k}$. This gives good bounds for $n=ck$ for  constant $c$, but already for $n=k^2$ the aforementioned ratio is a constant independent of $k$.\\

\textbf{Remark.} Very recently, Kaiser and Stehl\'{\i}k \cite{KS} constructed edge-critical (w.r.t. chromatic number) subgraphs of $SG_{n,k}$ with few edges. Their construction, using the argument above, implies that $\zeta\ge 2^{-k}{2k\choose k} = (2+o(1))^k.$ We  discuss these graphs in Section~\ref{secxg}. \\

It may be even more natural to study the vertex version of the problem. We believe that the following strengthening of  Lov\'asz' bound $\chi(KG_{n,k})\ge n-2k+2$ should be true.

\begin{conj}\label{conj1} The largest subset of vertices of $KG_{n,k}$ that may be properly colored in $n-2k+1$ colors has size at most ${n\choose k}-c^k$, where $c>1$ is some constant.
\end{conj}
Note that Conjecture~\ref{conj1} deals with a particular instance of Problem~\ref{pr1}.
If one modifies the canonical coloring of $KG_{n,k}$ by first taking $n-2k$ stars and then an intersecting family on the remaining set ${X\choose k}$, $|X|=2k$, then the number of ``missing'' sets is $\frac12{2k\choose k}\approx 4^k$. Interestingly, we can do better. Assume that $n=3k$ and color a set $A\in {[n]\choose k}$ into the $i$-th color if $|A\cap [3i-2,3i]|\ge 2$, $i=1,\ldots, k$. Clearly, each color is an intersecting family, moreover, we used $k=n-2k$ colors. At the same time, a set $B$ is not colored if and only if for each $i\in[k]$ it satisfies $|B\cap [3i-2,3i]|=1$. There are $3^k$ such sets in total. One can do the same for larger $n$ by first taking the stars and using this coloring on the last $3k$ elements. We actually believe that this is essentially the best one can do, and, in particular, $c=3+o_k(1)$ in the conjecture above. We note that the same argument with averaging over different copies of Schrijver graphs would also work here, but again has the same limitations (i.e., gives interesting results only for small $n = n(k)$).

The following is based on the discussion we had with Florian Frick and G\'abor Tardos. Let us first give some topological preliminaries. (For those unfamiliar with the subject, we advise  to consult the book of Matou\v sek \cite{Mat}.) For a family $\ff$, let $KG(\ff)$ be the {\it Kneser graph of $\ff$}, that is, the graph with vertex set $\ff$ and edges connecting disjoint sets. In particular, $KG_{n,k}=KG({[n]\choose k})$. 
A {\it simplicial complex} $\mathcal H\subset 2^{[n]}$ is a family satisfying the condition that if $H\in \mathcal H$ and $H'\subset H$, then $H'\in \mathcal H$. For a simplicial complex $\mathcal H$, we denote the union of all its simplices (its {\it polyhedron}) by $\|\mathcal H\|$. One can define the {\it deleted join} $\mathcal H^{*2}_{\Delta}$ of a complex $\mathcal H$ as follows:
$$\mathcal H^{*2}_{\Delta}:=\big\{(H_1\times\{1\})\cup (H_2\times \{2\}):H_1,H_2\in \mathcal H, H_1\cap H_2=\emptyset\big\}.$$
There is a natural free $\mathbb Z_2$-action on $\mathcal H^{*2}_{\Delta}$, and we can define the $\mathbb Z_2$-index $\mathrm{ind}_{\mathbb Z_2}(\mathcal H^{*2}_{\Delta})$ of $\mathcal H^{*2}_{\Delta}$ as the minimum dimension $d$ of a sphere $S^{d}$, for which there exists a continuous map $\|\mathcal H^{*2}_{\Delta}\|\to S^{d}$ that commutes with the $\mathbb Z_2$-actions on the spaces.
The following theorem was proven by Sarkaria (see \cite[Theorem~5.8.2]{Mat}):

\begin{thm}\label{thmsark}
  Let $\mathcal H$ be a simplicial complex on $n$ vertices and let $\ff$ be the family of all inclusion-minimal sets in $2^{[n]}\setminus \mathcal H$ (the {\it family of minimal non-faces of $\mathcal H$}). Then
  \begin{equation}\label{eqsark}\mathrm{ind}_{\mathbb Z_2}(\mathcal H^{*2}_{\Delta})\ge n-\chi(KG(\ff))-1.\end{equation}
  Consequently, if $d\le n-\chi (KG(\ff))-2$ then, for any continuous map $f: \|\mathcal H\|\to \mathbb R^{d}$, the images of some two disjoint faces of $\mathcal H$ intersect.
\end{thm}

Consider a family $\ff\subset{[n]\choose k}$ that is a union of $n-2k+1$ intersecting families and let $\mathcal H\subset 2^{[n]}$ be the family of sets in $2^{[n]}$ that do not contain any $F\in \ff$. Clearly, $\mathcal H$ is a simplicial complex, moreover, $\ff$ is the family of minimal non-faces of $\mathcal H$. We have $\chi (KG(\ff))\le n-2k+1$ and thus $\mathrm{ind}_{\mathbb Z_2}(\mathcal H^{*2}_{\Delta})\ge 2k-2$ by Theorem~\ref{thmsark}. Thus, Conjecture~\ref{conj1} is implied by the following conjecture.
\begin{conj}\label{conj2}
There exists $c>1$, such that for every $k$ if $\mathcal H$ is a simplicial complex that has fewer than $c^k$ $k$-element sets, then $\mathrm{ind}_{\mathbb Z_2}(\mathcal H^{*2}_{\Delta})\le 2k-3$ (or even the following is true: $\mathcal H$ is embeddable into $\mathbb R^{2k-3}$).
\end{conj}
Note that, substituting $c=1$ in the conjecture above, we get the Geometric Realization Theorem, stating that any finite $(k-2)$-dimensional simplicial complex has a geometric realization in $\mathbb R^{2k-3}$ (see \cite[Theorem~1.6.1]{Mat}).

\section{Proofs of lower bounds from Theorems~\ref{thm1} and~\ref{thm3}}\label{sec3}

For shorthand, we say that a subgraph $H$ of $KG_{n,k}$ or $KG_{n,k}(p)$ is {\it induced on a subset $R\subset[n]$} instead of saying that $H$ is an induced subgraph of the corresponding graph on the subset ${R\choose k}$ of vertices. We denote such graph by $KG_{n,k}|_R$ and $KG_{n,k}(1/2)|_R$, respectively.
The {\it star} $\mathcal S_x$ with center $x$, $x\in[n]$, is the collection of all vertices of $KG_{n,k}$ which contain  $x$.



In what follows, we say $f(n) \gg g(n)$ iff $g(n) = o(f(n))$. To simplify presentation, we omit floors and ceilings in the expressions that are meant to be integral in the places where this does not affect the calculations.

\medskip

Bound \eqref{eqgrowk} implies the statement of Theorem~\ref{thm3} for any $k\ge C(n\log n)^{1/6}$. Moreover,  $k$ is fixed and $n$ tends to infinity in Theorem~\ref{thm1}. Therefore, in what follows, we assume that $n \gg k^5$ for the proof of both theorems. In this section, we give the proofs of Theorems~\ref{thm1} and~\ref{thm3} modulo three key lemmas (Lemmas~\ref{lem2} and~\ref{lem3} for Theorem~\ref{thm1} and Lemmas~\ref{lem2} and~\ref{lem3_growk} for Theorem~\ref{thm3}). We give the proofs of the lemmas in the latter subsections.

\begin{lemma}\label{lem2}
    Let $h, k\ge 2$ satisfy $h \gg k^{5}$,  ${h / (3k^2)\choose k} > 2^{20k} h\log_2 n$ and ${h / (3k^2)\choose k}^2 > 2^{40k} h^3\log_2 n$. Then the random graph $KG_{n,k}(1/2)$ a.a.s. satisfies the following condition, which holds for all  $h$ and $G$ as below simultaneously.

    If $G$ is a subgraph of $KG_{n,k}(1/2)$, induced on a subset of  $h$ elements from $[n]$ and it is properly colored in $t \le h$ colors, then one of the colors forms a subset of a star.
\end{lemma}

For constant $k \ge 2$, put $h_0:= C \sqrt[k-1]{\log_2 n},$ and for $k > \log_2\log_2 n$ put $h_0 := k^{5.1}$ (here, the two regimes of $k$ correspond to the two theorems). It is not difficult to verify that $h:=h_0$ satisfies the conditions imposed on $h$  in Lemma~\ref{lem2} for the corresponding $k$, provided that $C>0$ is sufficiently large.


We proceed with the proofs of the theorems. Consider a coloring of $KG_{n,k}(1/2)$ into $\chi:=\chi(KG_{n,k}(1/2)) < n$ colors $C_1, \ldots, C_\chi$.
Applying the conclusion of Lemma~\ref{lem2} with $h=n$ and $KG_{n,k}(1/2)$ itself playing the role of $G$, we get that one of the colors, say $C_1$, is a subset of a star with, say, center in $x_{1}$. Thus, $C_1$ is not used in the coloring of $KG_{n,k}(1/2)|_{S_1}$, where $S_1:=[n]-\{x_1\}$. Next, apply the conclusion of Lemma~\ref{lem2} with $h=n-1$ and the subgraph $KG_{n,k}(1/2)|_{S_1}$ playing the role of $G$. We get that some color $C_2$ in the coloring of $KG_{n,k}(1/2)|_{S_1}$ is a subset of a star with center in $x_2$ and thus is not used in the coloring of $KG_{n,k}(1/2)|_{S_2}$, where $S_2:=S_1-\{x_2\}$. 
We can repeat this argument as long as $h\ge h_0$, i.e., as long as the conditions on $h$ in Lemma~\ref{lem2} are satisfied. Lemma~\ref{lem2} states that a.a.s. the conclusions of all of these steps hold simultaneously, and that in particular a.a.s. there exist colors $C_1,\ldots, C_{n-h_0}$ and a set of vertices $X:=[n]-\{x_1,\ldots, x_{n-h_0}\}$ such that only the colors $C_{n-h_0+1},\ldots, C_\chi$ are used in the coloring of $KG_{n,k}(1/2)|_{X}$. In other words, a.a.s., there exists a set $X\subset [n]$, $|X| = h_0$, such that $KG_{n,k}(1/2)|_X$ is properly colored using $\chi-n+h_0$ colors. In what follows, we restrict our attention to such set $X$ and $KG_{n,k}(1/2)|_X$.


We first consider the case of constant $k$.

\begin{lemma}\label{lem3}
    Fix $C > 0$ and integer $k$. The graph $KG_{n,k}(1/2)$ a.a.s. satisfies the following property simultaneously for all $G$ as below.

    If $G$ is a subgraph of $KG_{n,k}(1/2)$, induced on $h_0 = C \sqrt[k-1]{\log_2 n}$ elements of $[n]$, then
    \begin{align*}
    \chi(G) >&~ h_0 - O_k\big(\sqrt[2k-2]{\log_2 n}\big) & \text{ if } k \ge 3,
    \\
    \chi(G) >&~ h_0 - O\big(\sqrt{\log_2 n \cdot \log_2\log_2 n}\big) & \text{ if } k = 2.
    \end{align*}
\end{lemma}

For $k\ge 3$, Lemma~\ref{lem3} implies that a.a.s. at least $h_0 - O_k\big(\sqrt[2k-2]{\log_2 n}\big)$ colors were used to color $KG_{n,k}(1/2)|_X$. Therefore, we a.a.s. get that 
$$\chi-n+h_0 = h_0 - O_k\big(\sqrt[2k-2]{\log_2 n}\big) \ \ \ \iff \ \ \ \chi = n - O_k\big(\sqrt[2k-2]{\log_2 n}\big).$$
The deduction for $k=2$ is similar. This concludes the proof of the lower bound in Theorem~\ref{thm1}.
To prove Theorem~\ref{thm3}, we need to use the following lemma instead of Lemma~\ref{lem3}.

\begin{lemma}\label{lem3_growk}
    Fix $\eps>0$. Then for any $k \ge (1+\eps)\log_2\log_2 n$, the graph $KG_{n,k}(1/2)$ a.a.s. satisfies the following property simultaneously for all $G$ as below.

    If $G$ is a subgraph of $KG_{n,k}(1/2)$, induced on $h_0 = k^{5.1}$ elements of $[n]$, then
    $$
    \chi(G) \ge h_0 - 2k - 10.
    $$
\end{lemma}




This concludes the proof of Theorem~\ref{thm3}. In the next three subsections, we prove Lemmas~\ref{lem2},~\ref{lem3} and~\ref{lem3_growk}, respectively.

\subsection{Proof of Lemma~\ref{lem2}}
We first summarize the proof of the lemma. It can be separated into three stages. 
In the first part, we show that any large independent set in  $KG_{n,k}(1/2)$ can contain only very few sets outside its largest star. 
To show that, we do certain estimates of the number of edges between sets of vertices in $KG_{n,k}$, in a way similar to how it was done in \cite{BNR, BBN}. 
In the second part, for any given set $B$ of $h$ vertices and its coloring $\mathcal M$, we find its subset that is almost completely colored using the parts of the colors that are stars. In the third part, we do an intricate counting argument, in which for each possible `forbidden' pair $(B,\mathcal M)$ we construct a certain witness of its `badness', called an imprint. One crucial twist that an imprint gives is that, when counting the possible number of imprints, we do not distinguish between the star parts of different colors, which, combined with the step, allows us to control the number of possible imprints very well. At the same time, we can still extract good bounds on the probability of a collection of sets being an imprint, coming from the fact that at least some edges that are formed on the imprint must vanish when passing to the random subgraph.
We conclude the proof by showing that the expected number of imprints tends to $0,$ and thus is a.a.s. is equal to $0$, which, in turn, implies that there are no `forbidden' pairs $(B,\mathcal M)$.\\

We start with some preliminary estimates. For a family $\mathcal{A} \subset \binom{[n]}{k},$ denote $d_x(\mathcal A)$ the degree of $x\in [n]$ in $\mathcal A$ and $\Delta(\mathcal A)$ the  maximum degree of $\mathcal A$. Define the {\it diversity} $\gamma(\mathcal A)$ of $\mathcal A$ by putting $\gamma(\mathcal A) :=|\mathcal A|-\Delta(\mathcal A)$.

\begin{prop}\label{prop1_growing_k}
    If for some $b \gg k$ a family $\mathcal A\subset{[b]\choose k}$ satisfies $|\mathcal A| \ge 2k\binom{b-2}{k-2}$, then the number $e(\mathcal A)$ of disjoint pairs of sets in $\mathcal A$ is at least $\min( \frac16 \gamma(\mathcal A)|\mathcal A|,  \frac{1}{144} |\mathcal A|^2 / \binom{2k}{k})\big)$.
\end{prop}

\begin{proof}
Let us consider two cases.\vskip+0.1cm

\textbf{Case I.}\
$\Delta(\mathcal{A}) \ge \frac23 |\mathcal A|$.\vskip+0.1cm

Considering the biggest star $\mathcal S\subset \mathcal A$, we get that each of the $\gamma(\mathcal A)$ sets of $\mathcal A\setminus \mathcal S$ is disjoint with all but at most $k\binom{b-2}{k-2}$ sets from $\mathcal S$. Using that $|\mathcal S| = \Delta$, we get
$$
e(\mathcal{A}) \ge
\gamma(\mathcal{A})\left(\Delta(\mathcal{A}) - k\binom{b-2}{k-2}\right) >
\frac16 \gamma(\mathcal A) |\mathcal A|.
$$

\textbf{Case II.}
$\Delta(\mathcal{A}) < \frac23 |\mathcal A|$.\vskip+0.1cm

For a family $\mathcal H \in \binom{[n]}{k}$, denote by $\ell(\mathcal H)$ the difference between the cardinality of $\mathcal H$ and that of the largest intersecting subfamily of $\mathcal H$. In \cite[Lemma~3.1]{BBN} it was shown that $e(\mathcal H) \ge \ell(\mathcal H)^2 / \binom{2k}{k}$. Also in \cite{Kup3} it was shown that there is an absolute constant $c > 0$ such that for an intersecting family $\mathcal F \in \binom{[b]}{k}$, $b > ck$, one has $\gamma(\mathcal F)\le {b-3\choose k-2}$.
In our case we have $b\gg k$, and thus the largest intersecting family $\mathcal F$ in $\mathcal A$ has size $\Delta(\mathcal F)+\gamma(\mathcal F)\le \Delta(\mathcal A)+{b-3\choose k-2}.$ We have
$$
\ell(\mathcal A) \ge |\mathcal A| - \Delta(\mathcal A) - \binom{b-3}{k-2} > \frac{1}{12}|\mathcal A|.
$$
Hence, by \cite[Lemma~3.1]{BBN}
$$
e(\mathcal A) > \frac{1}{144} \frac{|\mathcal A|^2}{\binom{2k}{k}}.
$$

\end{proof}

\begin{prop}\label{prop2_growing_k}
    Let $k = k(n)$ satisfy $k \ge 2$.
    $KG_{n,k}(1/2)$ a.a.s. satisfies the following.

    For any subgraph $G$  of $KG_{n,k}(1/2)$ induced on a set of $b$ elements from $[n]$ and any independent set $\mathcal{A}$ in $G$, the number $e(\mathcal A)$ of  pairs of disjoint sets in $\mathcal A$ is at most $2(|\mathcal A| k\log_2 b + b\log_2 n)$.

    In particular, if $|\mathcal A| \ge \frac1{2b}\binom{b}{k}$, $\frac1{2b}\binom{b}{k} > 2^{19k}\log_2 n$ and $b \gg k^3$, then $\gamma(\mathcal A) < 12k\log_2b+ 2^{-16k} b$.
\end{prop}

\begin{proof}
    Denote $X_{b, N}$ the number of pairs $(B, \mathcal{A})$, where $B\in{[n]\choose b}$, $\mathcal{A}\subset\binom{B}{k}$, $|\mathcal{A}|=N$, $e(\mathcal{A}) > 2(|\mathcal A| k\log_2 b + b\log_2 n)$ and $\mathcal{A}$ is an independent set in the subgraph $G$ of $KG_{n,k}(1/2)$, induced on the set $B$. Then we get
$$
\E X_{b,N} \le
\binom{n}{b} \binom{\binom{b}{k}}{N} 2^{-e(\mathcal A)} \le
2^{b\log_2 n + Nk\log_2 b - e(\mathcal A)}
$$
$$
\le
2^{-Nk\log_2 b -b\log_2 n}.
$$

To finish the proof of the first part of Proposition~\ref{prop2_growing_k}, it is sufficient to show that
$$
\Prb:=\Pr(\exists b\ge 1\ \exists N \ge 1\colon X_{b,N} > 0) \to 0\ \ \text{as} \ \
n\to\infty.
$$

Indeed,
$$
\Prb\le
\sum_{b=1}^n \sum_{N=1}^{\binom{b}{k}} \E X_{b,N} \le
\sum_{b=1}^n \sum_{N=1}^{\binom{b}{k}} 2^{-Nk\log_2 b - b\log_2 n} =
\sum_{b=1}^n O\big(2^{-b\log_2 n}\big) = O\big(2^{-\log_2 n} \big) \to 0.
$$
The second part of the statement follows from Proposition~\ref{prop1_growing_k}. From the first part of Proposition~\ref{prop2_growing_k} we have, under the imposed conditions and sufficiently large $n$, $e(\mathcal A) < 2(|\mathcal A|k\log_2 b + b\log_2 n)$, which is at most $2|\mathcal A|\left(k\log_2 b + 2^{-19k}b\right)$ by the assumption on $|\mathcal A|$. At the same time, the assumptions $|\mathcal A|\ge \frac 1{2b}{b\choose k} = \frac 1{2k}{b-1\choose k-1}$ and  $b \gg k^3$ imply $|\mathcal A|\ge 2k{b-2\choose k-2}$, and thus we can apply  Proposition~\ref{prop1_growing_k}. We have
$$
\min\left(\frac16 \gamma(\mathcal A), \frac{b}{4000} \right) \le \min\left(\frac16 \gamma(\mathcal A), \frac{|\mathcal A|}{144{2k\choose k}} \right)  \le \frac{e(\mathcal A)}{|\mathcal A|} < 2k\log_2b + 2^{-18k}b,
$$
where the first inequality uses that $\frac {|\mathcal A|} {144 {2k\choose k}}\ge \frac b{4000}$ for any $k\ge 2$. 
Since $2k\log_2b + 2^{-18k}b < b/4000$, we get $\frac 16 \gamma(\mathcal A)<2k\log_2 b+2^{-18k}b,$ which implies the second part of the statement.
\end{proof}


Now we are ready to proceed with the proof of Lemma~\ref{lem2}. Consider a pair $(B, \mathcal{M})$, where $B\in{[n]\choose h}$, $h \gg k^5$ and $\mathcal{M}=\{\mathcal A_1,\ldots, \mathcal A_t\}$ is a proper coloring of $KG_{n,k}(1/2)|_B$ into $t\le h$ colors, none of which is a subset of a star.\footnote{We think of $\mathcal A_i$ as of a family of sets.} By possibly adding empty colors, we can w.l.o.g. assume that $t=h$. We call such a pair $(B, \mathcal{M})$ {\it a forbidden partition}. 

Informally, the next step in the proof is to find a large subset of $B$ that is almost completely covered by the parts of the colors that are stars. Recall that $\mathcal S_{j}:=\{S\in {[n]\choose k}: j\in S\}$.

\begin{prop}\label{prop_imprint_construction} A.a.s., for every such $B$ there exists a set $B'\subset B$ of cardinality $m\ge h/4$ and a renumbering of the colors from $\mathcal M$, such that $$\bigcup_{j\in B'} \Big|\mathcal S_j\cap \mathcal A_j\cap {B'\choose k}\Big|\ge \big(1-2^{-16k}\big){m\choose k}.$$
\end{prop}

\begin{proof}
Put $b := \frac{h}{3k^2}$.
Note that from the statement of the lemma we have $b \gg k^3$ and $\binom{b}{k} > 2^{20k} h\log_2 n$. W.l.o.g., we assume that $B = [h].$

Take ${[h]\choose k}$ and consider the largest color, w.l.o.g., $\mathcal A_{h}$. It has size at least $\frac 1h{h\choose k}$.  Suppose that the center of the biggest star in $\mathcal A_{h}$ is $h$.
Remove $h$ from the ground set and 
repeat the same procedure  on ${[i]\choose k}\setminus \mathcal (\mathcal A_{i+1}\cup\ldots\cup \mathcal A_h)$ for each $i=h-1,h-2\ldots,b+1$. At step $i$ we select a family $\mathcal A_i$ such that  $\big|\mathcal A_i\cap {[i]\choose k}\big|\ge \frac 1i\big({i\choose k}-\sum_{j=i+1}^h \gamma(\mathcal A_j)\big)$ and w.l.o.g. assume that the center of $\mathcal A_i\cap {[i]\choose k}$ is in $i$. 
By Proposition~\ref{prop2_growing_k}, we have $\gamma(\mathcal A_{j}\cap {[j]\choose k}) \le 12k\log_2h+2^{-16k}h$, provided $|\mathcal A_j\cap {[j]\choose k}|\ge \frac 1{2i}{i\choose k}.$ If this holds for any $j>i$ then we have $$\sum_{j=i+1}^h \gamma(\mathcal A_j) \le h(12k\log_2h+2^{-16k}h)=O(bk^3\log_2 h)+2^{-16k}9k^4 b^2\le O(bk^3\log_2 b)+2^{-12k}b^2.$$ 
In the last inequality, we used $2^{4k} > 9k^4$ and $2\log_2 h\ge \log_2 b.$ We verify the inequality $|\mathcal A_i|\ge \frac 1{2i}{i\choose k}$ by reverse induction on $i$, starting from $i=h$ (in which case it is straightforward), and we assume that it holds for $j>i.$ Applying the displayed inequality above, we get that (after  a suitable reordering of the families $\mathcal A_i$ and the elements of $[n]$) for each $i \in [b+1,h]$ we have
$$
\left|\mathcal A_i \cap \mathcal S_{i} \cap \binom{[i]}{k}\right| \ge
\frac1i \left( \binom{i}{k} - O(k^3b\log_2b)- 2^{-12k}b^2 \right) \ge
\frac1{2i} \binom{i}{k}.
$$

For $i\in [b+1,h]$, consider the family $\mathcal A_i$. We know that all but very few subsets from the family $\mathcal A_i \cap \binom{[i]}{k}$ belong to the star $\mathcal S_{i}$. Let us show that the same holds for the family $\mathcal A_i$ itself. Define $\mathcal A_i' := (\mathcal A_i \setminus \mathcal S_{i}) \cup \left(\mathcal A_i \cap \mathcal S_{i} \cap \binom{[i]}{k}\right) \subset \mathcal A_i$. As in Case I of the proof of Proposition~\ref{prop1_growing_k}, we have
\begin{equation*}
e(\mathcal A_i') \ge
|\mathcal A_i \setminus \mathcal S_{i}| \left( \Big|\mathcal A_i \cap \mathcal S_{i} \cap \binom{[i]}{k}\Big| - k\binom{i - 2}{k - 2} \right)
=
(1 + o(1)) |\mathcal A_i \setminus \mathcal S_{i}| \Big|\mathcal A_i \cap \mathcal S_{i} \cap \binom{[i]}{k}\Big|.
\end{equation*}
Note that the equality above is due to $i>b=\frac h{3k^2} \gg k^3$.
Since the family $\mathcal A_i'$ is independent in $KG_{n,k}(1/2)$, from Proposition~\ref{prop2_growing_k} we get
\begin{align}
|\mathcal A_i \setminus \mathcal S_{i}| \le&\
2(1+o(1))\frac{|\mathcal A_i'|k\log_2 h + h\log_2 n}{\left|\mathcal A_i \cap \mathcal S_{i} \cap \binom{[i]}{k}\right|}
\nonumber \\
\le&\
2(1+o(1))\left(   \frac{\left|\mathcal A_i \cap \mathcal S_{i} \cap \binom{[i]}{k}\right| + |\mathcal A_i \setminus \mathcal S_{i}|}{\left|\mathcal A_i \cap \mathcal S_{i} \cap \binom{[i]}{k}\right|}k\log_2 h + \frac{h\log_2 n}{\frac1{2i}\binom{i}{k}} \right)
\nonumber \\
=&\ 
2(1+o(1))\left(k\log_2 h + o(|\mathcal A_i \setminus \mathcal S_{i}|)+ \frac{h\log_2 n}{\frac1{2i}\binom{i}{k}}\right). \label{out_of_star}
\end{align}
Note that $2k\log_2 h = o\big(\frac 1{2^{40k}h}{h\choose k}\big)$. Next, from the lemma statement we have
$\binom{i}{k}^2 \ge \binom{b}{k}^2 > 2^{40k} h^3 \log_2 n$, and thus
$$
\frac{h\log_2 n}{\frac1{2i}\binom{i}{k}}\le \frac{{i\choose k}^2}{2^{40k}h^3\log_2 n}\frac{h\log_2 n}{\frac1{2i}\binom{i}{k}} \le \frac{2{i\choose k}}{2^{40k}h}\le \frac{1}{2^{39k}h}{h\choose k}.
$$
Combining these  with \eqref{out_of_star}, we get that
\begin{equation}\label{eqminus1}
    |\mathcal A_i \setminus \mathcal S_{i}| <\frac{1}{2^{38k}h}\binom{h}{k}.
\end{equation}

Now consider a family $\mathcal A_i$ such that $i\in[b]$ and $|\mathcal A_i| > \frac{1}{2^{20k}h}\binom{h}{k}$.
We aim to find few elements of the ground set which altogether cover most of the subsets of the family $\mathcal A_i$.  Define the set $D(\mathcal A_i) := \{x\in[n]\colon d_x(\mathcal A_i) > \frac{|\mathcal A_i|}{2k}\}$. Note that $|D(\mathcal A_i)| \le 2k^2$. Every subset from $\binom{[n]\setminus D(\mathcal A_i)}{k}$ intersects at most $k \cdot \frac{|\mathcal A_i|}{2k}$ subsets from $\mathcal A_i$. Therefore we have
$$
e(\mathcal A_i) \ge
\left| \mathcal A_i \cap \binom{[n]\setminus D(\mathcal A_i)}{k} \right| \frac{|\mathcal A_i|}{2}.
$$

First using Proposition~\ref{prop2_growing_k} and then the inequalities $|\mathcal A_i| > \frac{1}{2^{20k}h}\binom{h}{k}$ and $\binom{h}{k}^2 > 2^{40k}h^3 \log_2 n$, we have
$$
\left| \mathcal A_i \cap \binom{[n]\setminus D(\mathcal A_i)}{k} \right| \le
4\frac{|\mathcal A_i|k\log_2 h + h\log_2 n}{|\mathcal A_i|} \le
4k\log_2 h + \frac{1}{2^{20k}h}\binom{h}{k} = (1+o(1))\frac{1}{2^{20k}h}\binom{h}{k}.
$$
Note that this inequality is trivially true for those $\mathcal A_i,$ $i\in [b],$ such that $|\mathcal A_i|\le \frac{1}{2^{20k}h}\binom{h}{k}$, and thus is true for any $i\in [b].$

Put $B':=[h]-([b]\cup D(\mathcal A_1)\cup \ldots\cup D(\mathcal A_b)).$ Note that $m := |B'| \ge h - b - 2bk^2 \ge \frac14 h$. Assume that $B' = \{c_1, \ldots, c_m\}$.  All but  $2^{-19k}\binom{h}{k}$ subsets from $\binom{B'}{k}$ are  in one of the colors $\mathcal A_{c_i},$ $1 \le i \le m$ and contain the corresponding element $c_i$. Indeed, any other set is either contributing to $\gamma(\mathcal A_i)$ for one of $i\in [b+1,h]$, or to $|\mathcal A_i\cap {B'\choose k}|$ for $i\le b.$  The number of the sets of the first type is small  due to \eqref{eqminus1}, and the number of the sets of the second type is small due to the last displayed inequality and the definition of $B'$.
\end{proof}

The last part of the proof of Lemma~\ref{lem2} is based on an intricate counting argument. The following notion is crucial for this counting. Consider a quadruple $(B, B', Z, \mathcal F)$ such that the following hold
\begin{itemize}
    \item $B \subset [n]$ and $|B| = h$.
    \item $B' = \{c_1, \ldots, c_m\} \subset B$ and $|B'| = m \ge \frac14 h$.
    \item $Z = (F_1, \ldots, F_m)$, where $F_i\in \binom{B}{k}$.
    \item $\mathcal F \subset \binom{B'}{k}$ and $\left| \binom{B'}{k} \setminus \mathcal F \right| < 2^{-16k}\binom{m}{k}$.
\end{itemize}

Let us estimate the number $M_{m}$ of such quadruples for a fixed $m$, $h/4 \le m \le h$.
\begin{align*}
M_{m} \le&
\sum_{h=m}^{4m} \sum_{|\mathcal F| = \binom{m}{k} -  2^{-16k}{m\choose k}}^{\binom{m}{k}}
    \binom{n}{h} \binom{h}{m} \binom{h}{k}^m \binom{\binom{m}{k}}{|\mathcal F|}\\
\le&\
6m \cdot\binom{n}{4m} \binom{4m}{m} \binom{4m}{k}^m \binom{\binom{m}{k}}{ 2^{-16k}{m\choose k}}\\
\le&\
2^{(1+o(1))4m\log_2 n + (1+o(1))H(2^{-16k})\binom{m}{k}} \le
2^{2^{-2k} {m\choose k}},
\end{align*}
where  $H(p)=-p\log_2 p-(1-p)\log_2 (1-p)$ is the binary entropy function and the last inequality is due to the condition  $\binom{m}{k} > 2^{20k} h\log_2 n$. We also used the easy-to-check inequality $H(2^{-16k})<2^{-3k}.$

We return to the proof of Lemma~\ref{lem2}. Consider a forbidden partition $(B, \mathcal M)$. We say that a quadruple $(B, B', Z, \mathcal F)$ as above is \emph{an imprint of} $(B, \mathcal M)$ if the following holds:
\begin{itemize}
    \item each $F_i \in Z$ is colored in color $\mathcal A_i$ and does not contain element $c_i$.
    \item each $F \in \mathcal F$ is colored in one of the colors $\mathcal A_i$, $1\le i \le m$, and contains the corresponding element $c_i$.
\end{itemize}

From Proposition~\ref{prop_imprint_construction}, each forbidden partition $(B, \mathcal M)$ has an imprint (up to a reordering of colors). 
We call a quadruple $(B, B', Z, \mathcal F)$ \emph{an admissible imprint} if it is an imprint for some forbidden partition $(B, \mathcal M)$. Then, in order to show that a.a.s. there are no forbidden partitions, it is enough to show that there are no admissible imprints.

Let us bound the probability that a quadruple $(B, B', Z, \mathcal F)$ with $|B'| = m$ is an admissible imprint. For that to happen, each of the $k$-sets $F$ in $\mathcal F$ in $KG_{n,k}(1/2)$ should satisfy the following: there is an $i$ such that, first, $c_i\in F$ and, second, there is no edge between $F$ and $F_i$ in the random graph.
We say that  $F$ is {\it good} if it is disjoint with all such $F_i$. For a good $F$ the probability of the event described before is $(1-2^{-k})$. Moreover, such events are independent for different good $F$. We have at most $k{m-2\choose k-2}$ sets from $\ff$ that intersect any given $F_i$ and contain $c_i$, and thus at least $|\ff|-mk{m-2\choose k-2}$ sets in $\ff$ that are good. Using that $|\ff|-mk{m-2\choose k-2}\ge (1-2^{-16k}-\frac{k^2(k-1)}{m-1}{m\choose k})>\frac 12{m\choose k}$, the probability $\Prb_m$ of the event that the imprint is admissible may be bounded as
$$
\Prb_m \le
(1 - 2^{-k})^{|\mathcal F| - mk\binom{m-2}{k-2}} \le
e^{-\left(|\mathcal F| - mk\binom{m-2}{k-2}\right) 2^{-k}} \le
e^{-\binom{m}{k} 2^{-k-1}}.
$$

Let $h'$ be the smallest $h$ that satisfies the conditions of Lemma~\ref{lem2}.
We denote by $X$ the random variable that counts the number of admissible imprints $(B, B', Z, \mathcal{F})$ with $m \ge h'/4$ in $KG_{n,k}(1/2)$. Then
\begin{align*}\Pr(X > 0) \le
\E X =
\sum_{m=h'/4}^n \Prb_m M_{m}\le
\sum_{m=h'/4}^n e^{-\binom{m}{k} 2^{-k-1}} 2^{2^{-2k}{m\choose k}}\\ \le 
\sum_{m=h'/4}^n e^{-\binom{m}{k} 2^{-k-2}} \le
e^{\ln n - \binom{h'/4}{k} 2^{-k-2}} \to 0 \ \  \text{as }\ n\to \infty.
\end{align*}

We conclude that a.a.s. $X = 0$, which concludes the proof of Lemma~\ref{lem2}.

\subsection{Proof of Lemma~\ref{lem3}}

Recall that $k$ is constant for this proof. For simplicity, let us assume that $G$ from the formulation of the lemma is induced on the set $B:=[h_0]$. The key to the proof of Lemma~\ref{lem3} is the following proposition, which allows us to estimate the number of monochromatic edges in any coloring of vertices of the subgraph of $KG_{n,k}$ induced on $B$, provided that the number of colors used is small. Fix a sufficiently large constant $C''>0$ and put $h_1:= C''\sqrt{\log_2 n \cdot \log_2 \log_2 n}$ for $k=2$ and $h_1:= C''\sqrt[2k-2] {\log_2 n}$ otherwise.

\begin{prop}\label{prop2}
    In the notations above, let $\mathcal{A}_1, \ldots, \mathcal{A}_\chi$ be a partition (coloring) of ${B\choose k}$, where $\chi = h_0 - h_1$, and let $0 < \eps' < 1$.

    Then there exist $C_2>0$, two disjoint subsets  $B', B'' \subset B$, $b := |B'| \ge h_1$, $l := |B''|$, $B'' = \{c_1, \ldots, c_l\}$ and a reordering of the families $\mathcal{A}_1, \ldots, \mathcal{A}_\chi$ such that for the family $$\mathcal F :=
    \binom{B'}{k} \cup  \bigcup_{i=1}^l \Big(\mathcal{A}_i\cap \mathcal S_{c_i}\cap \binom{B'\cup B''}{k}\Big)$$ 
the total number 
    of monochromatic edges in the subgraph of $KG_{n,k}$ induced on $\mathcal F$ is at least $M(2):=C_2 \frac{(b+l) b^{2}}{\log_2\log_2 n}$ for $k=2$ and at least $M(k):=C_2 (b+l)^{k-1 - \eps'} b^{k+\eps'}$ otherwise.
\end{prop}

\begin{proof}
Put $C_1 := 6k\cdot k!$. We need to classify $\mathcal A_i$ into two different types:  the ones that have ``big stars'' and the ones that do not. To do so, on the $i$-th step, $i=1,\ldots$ we find a family with a star of size at least $\frac{(h_0-i+1)^{k-1}}{C_1}$ and remove the corresponding center of the star from the ground set and the family from the list of colors. We repeat until there is no family with such a star. Reordering the families $\mathcal A_i$ and the elements of $B$, we conclude that for some  $b \ge h_1$ the following holds:

1. For $i=1,\ldots, h_0-b$ we have $\left|\mathcal{A}_i \cap \mathcal{S}_{b+i} \cap \binom{[b+i]}{k}\right| \ge \frac{(b+i - 1)^{k-1}}{C_1}$.

2. For $h_0-b+1\le i\le \chi$ and any $x \le b$, we have $\left|\mathcal{A}_i \cap \mathcal{S}_{x} \cap \binom{[b]}{k}\right| < \frac{(b - 1)^{k-1}}{C_1}$.\\

That is, $\mathcal{A}_{1}, \ldots, \mathcal{A}_{h_0-b}$ are the families (colors) with $C_1$-fraction stars with centers in $b+1, \ldots, h_0$, respectively, and 
the families $\mathcal{A}_{h_0-b+1}, \ldots, \mathcal{A}_{\chi}$ do not have positive-fraction stars when restricted to  $\binom{[b]}{k}$.

Put
\begin{align*}
L_1 :=& \sum_{i=1}^{h_0-b} \left|\mathcal{A}_i \cap \binom{[b]}{k}\right|,\\
L_2 :=& \sum_{i=h_0-b+1}^{\chi} \left|\mathcal{A}_i \cap \binom{[b]}{k}\right|.
\end{align*}
Since $L_1 + L_2 = \binom{b}{k}$, we either have $L_1 \ge \frac12 \binom{b}{k}$, or $L_2 \ge \frac12 \binom{b}{k}$.

\medskip

\textbf{Case I.\ \  $L_1\ge \frac12 \binom{b}{k}$.}
\vskip+0.1cm
For integer $s\ge 0$ put $$\mathcal C_s:=\bigcup_{i=b 2^s+ 1}^{\min(b 2^{s+1}, h_0)} \mathcal A_{i-b}.$$
Note that starting from some $s$, the families $\mathcal C_s$ are empty. Clearly, $\sum_{s=0}^{\infty} \left|\mathcal C_s \cap {[b]\choose k}\right| = L_1.$
Since $\sum_{s=0}^{\infty} 2^{-\eps' s} = \frac {1}{1-2^{-\eps'}}$, one can find $s \ge 0$, such that
$$
\left|\mathcal C_s \cap {[b]\choose k}\right| \ge (1 - 2^{-\eps'}) 2^{-\eps' s}L_1.
$$

Each set from $\mathcal{A}_i \cap \binom{[b]}{k}$, where $1\le i \le h_0-b,$ is disjoint with at least $\frac{(b+i-1)^{k-1}}{C_1} - k\binom{b+i - 2}{k-2} = (1 + o(1)) \frac{(b+i-1)^{k-1}}{C_1}$ sets in $\mathcal{A}_i \cap \mathcal S_{b+i} \cap \binom{[b+i]}{k}$. Therefore, each set from $\mathcal C_s\cap {[b]\choose k}$ forms at least $(1+o(1))\frac {(b2^s)^{k-1}}{C_1}$  {\it monochromatic edges} with sets from $\mathcal C_s\cap {[b2^{s+1}]\choose k}$.
So the family $\mathcal C_s\cap {[b2^{s+1}]\choose k}$  induces at least 
$$
(1 + o(1))\frac{(b2^s)^{k-1}}{C_1} \cdot (1 - 2^{-\eps'}) 2^{-\eps' s}L_1 \ge
C_1'(b2^{s})^{k-1-\eps'}b^{\eps'}{b\choose k}=:(*)
$$ 
monochromatic edges for some constant $C_1'>0$. Let us put $B':=[b], B'':=[b+1, b2^{s+1}]$ and rewrite the quantity above.

$$(*)\ge C_1''\frac{(b2^{s+1})^{k-1-\eps'} b^{k+\eps'}}{2^{k-1-\eps'} k!} \ge
C_2(|B'| + |B''|)^{k-1-\eps'} {|B'|}^{k+\eps'},
$$
where $C''_1, C_2>0$ are again some constants depending on $k$. Thus, the proposition holds for the $B'$ and $B''$ as above and $k\ge 3$. 

If $k=2$ then we first  find $0 \le s \le \log_2 (h_0/ b) = \Theta(\log_2\log_2 n)$ such that
$$
\left|\mathcal C_s \cap {[b]\choose k}\right| \ge 
\frac{L_1}{\Theta(\log_2\log_2 n)}
$$
using simple pigeon-holing. And then the rest of the calculations stays the same.
\medskip

\textbf{Case II.\ \ $L_2 \ge \frac12 \binom{b}{k}$.}\vskip+0.1cm

Let us put $q:=\chi-(h_0-b)$. That is, $q$ is the total number of families (colors) $\mathcal A_i$ that contribute to $L_2$. Let us call a family $\mathcal{A}_i$, $\chi-q+1\le i \le \chi$, {\it big}, if
$|\mathcal{A}_i \cap \binom{[b]}{k}| \ge \frac{L_2}{2q}$. The total contribution of small $\mathcal A_i$ to $L_2$ is at most $q \frac{L_2}{2q}$. Therefore, at least a half of the sets contributing to $L_2$ lie in big families. In what follows, we restrict our attention to big families.

Since for any $\chi-q + 1 \le i \le \chi$ and $1 \le x \le b$ we have $|\mathcal{A}_i \cap \mathcal{S}_{x} \cap \binom{[b]}{k}| < \frac{(b - 1)^{k-1}}{C_1}$ (no big stars condition), any set from a big family $\mathcal{A}_i \cap \binom{[b]}{k}$ is disjoint with at least $\frac {L_2}{2q}- k \frac{(b - 1)^{k-1}}{C_1}$ sets in that family. Using the condition on $L_2$ and the fact that $q\le b$, we get that
$$
\frac {L_2}{2q}- k \frac{(b - 1)^{k-1}}{C_1}\ge \frac{\binom{b}{k}}{4b} - k \frac{(b - 1)^{k-1}}{C_1} =
\frac{\binom{b}{k}}{4b} \left(1 - \frac{4kb(b-1)^{k-1}}{6k\cdot k! \binom{b}{k}}\right) \ge
 (1+o(1))\frac{\binom{b}{k}}{12b}.
$$

Thus, the restrictions of the sets  $\mathcal{A}_{\chi-q+1}, \ldots, \mathcal{A}_\chi$ on $\binom{[b]}{k}$ induce at least
$$(1+o(1))\frac{\binom{b}{k}}{12b} \cdot \frac{L_2}{2}\ge C_2b^{2k-1}\ge M(k)$$ edges for some positive constant $C_2$. The proposition holds for $B' := [b]$ and $B'' = \emptyset$.
\end{proof}

\medskip

Let us prove Lemma~\ref{lem3} using Proposition~\ref{prop2}. Assume that $\chi(G) \le h_0 - h_1$ and let  $\mathcal{A}_1, \ldots, \mathcal{A}_{h_0 - h_1}$ be a proper coloring of the vertices of  $G$ in $h_0 - h_1$ colors. From Proposition~\ref{prop2}, we conclude that there exist disjoint subsets $B'$, $B''$, and a reordering of the families $\mathcal{A}_1, \ldots, \mathcal{A}_{h_0 - h_1}$, such that the family $\binom{B'}{k} \cup (\mathcal{A}_1\cap \mathcal S_{c_1}\cap \binom{B'\cup B''}{k})\cup\ldots \cup (\mathcal{A}_l \cap \mathcal S_{c_l}\cap \binom{B'\cup B''}{k})$ induces at least $M(k)$ edges in $KG_{n,k}$, which must be absent in $G$. The probability $\Prb$ of such an event, over all possible choices of $B'$, $B''$, $\mathcal A_i \cap \mathcal S_{c_i}\cap \binom{B'\cup B''}{k}$ and colorings of vertices of $k$-sets on $B'$, satisfies
$$
\Prb \le
\sum_{b = h_1}^{h_0}\sum_{l=0}^{h_0-b} \binom{n}{b}\binom{n-b}{l} l! (h_0 - h_1)^{\binom{b}{k}} 2^{l \binom{b+l}{k-1}} 2^{- M(k)} \le
\sum_{b = h_1}^{h_0}\sum_{l=0}^{h_0-b} 2^{(b+l) \log_2 n + b^k \log_2 h_0 + (b+l)^{k} -M(k)} .
$$

Consider the case $k \ge 3$. Recall that $M(k)=C_2 (b+l)^{k-1-\eps'} b^{k+\eps'}$.
Thus, for sufficiently large $C''$ (from the definition of $h_1$) we have
\begin{align*}
(b + l) \log_2 n =&\
(1 + o(1)) (b + l) C''^{-(2k-2)} h_1^{2k-2} \le
(1 + o(1)) C''^{-(2k-2)} (b+l)^{k-1-\eps'} b^{k+\eps'} \\
<&\
\frac12 C_2 (b+l)^{k-1-\eps'} b^{k+\eps'} = M(k)/2,
\\
(b + l)^{k}\log_2 h_0 \le&\
(b + l)^{k-1-\eps'} h_0^{1+2\eps'} \le
(b + l)^{k-1-\eps'} h_1^{2(1+2\eps')} \\
=&\
o\left((b + l)^{k-1-\eps'} b^{k+\eps'} \right)= o(M(k)).
\end{align*}

Consider the case $k=2$. Recall that $M(2)=C_2 \frac{(b+l) b^{2}}{\log_2\log_2 n}$ and that $b + l \le h_0 = C \log_2 n$.
For sufficiently large $C''$ the quantity $\max\{(b+l)\log_2 n, (b+l)^2\}$  can be bounded by
\begin{align*}
\max \{1,C\} (b + l) \log_2 n =&\
(1 + o(1)) \max\{1,C\} C''^{-2} (b + l) \frac{h_1^{2}}{\log_2 \log_2 n} <
\frac14 C_2 \frac{(b+l) b^{2}}{\log_2\log_2 n} = \frac {M(2)}4.
\end{align*}
Moreover, $b^2\log_2h_0<b^{2.1}=o(M(2)).$
Thus both for $k\ge 3$ and $k=2$ we have 
$$
\Prb \le
\sum_{b = h_1}^{h_0}\sum_{l=0}^{h_0-b} 2^{(b + l) \log_2 n + b^k \log_2 h_0 + (b+l)^{k} - M(k)} \le
\sum_{b = h_1}^{h_0}\sum_{l=0}^{h_0-b} 2^{ - \frac13 M(k)} \to 0.
$$

\subsection{Lemma~\ref{lem3_growk}}

\subsubsection{Sparse subgraphs of $KG_{n,k}$ with the same chromatic number}
\label{secxg}


In \cite{KS} Kaiser and Stehl\'{\i}k constructed a subgraph $XG_{n, k}$ of $SG_{n, k}$ such that $\chi(XG_{n, k}) = n - 2k + 2$ and any proper subgraph of $XG_{n, k}$ has smaller chromatic number. In other words, while $SG_{n, k}$ is a vertex-critical subgraph of $KG_{n, k}$, $XG_{n, k}$ is an edge-critical subgraph of $KG_{n, k}$. We note that we do not use the criticality of $XG_{n,k}$. What is of importance for us is that $XG_{n,k}$ always has much fewer edges than $KG_{n,k}.$ In what follows, we use a certain supergraph of $XG_{n,k}$ to prove the following lemma.

\begin{lemma}\label{lemkai}
In any coloring of $KG_{n,k}$ with $n-2k+1$ colors there are at least $2^{-k}{2k\choose k}$ monochromatic edges. 
\end{lemma}
Note that $2^{-k}{2k\choose k}>2^{k-1}/k.$
\begin{proof}
We start by defining a graph $XG'(n,k)$ that contains $XG(n,k)$ as a subgraph. In this proof we order elements of $[n]$ on a cycle in the clockwise direction and perform addition modulo $n$.
For $a, b\in [n]$, $[a, b]$ denotes the set $\{a, a+1, \ldots, b\}$, that is, a segment on the circle from $a$ to $b$ in the clockwise direction.

For a given edge $(F_1, F_2)$ of $KG_{n, k}$ we call an interval $[d, c]$, $c \le k < d$, \emph{admissible}, if $|[d, c]\cap F_1| = |[d, c]\cap F_2| = c$. Note that for two admissible intervals $[d_1, c]$ and $[d_2, c]$ we have $[d_1, c]\cap F_1 = [d_2, c]\cap F_1$ and $[d_1, c]\cap F_2 = [d_2, c]\cap F_2$, since either $[d_1, c]\subset [d_2, c]$ or $[d_2, c] \subset [d_1, c]$.

For an admissible interval $[d, c]$ we define  \emph{switching} at $[d, c]$. It maps a given edge $(F_1, F_2)$ of $KG_{n,k}$ to an edge $(F_1', F_2')$ via the following rule:
$$
F_1' := F_1~\triangle~([d, c] \cap (F_1\sqcup F_2)),
$$
$$
F_2' := F_2~\triangle~([d, c] \cap (F_1\sqcup F_2))
$$
(that is, it swaps elements of $F_1$ and $F_2$ inside $[d, c]$).
Note that for two admissible intervals $[d_1, c]$ and $[d_2, c]$ the switching operation will be the same. Also note that switching operations are commutative and that if apply switching at $[d, c]$ to an edge $(F_1, F_2)$ two times, we get the same edge again. All these properties imply that for a given edge $(F_1, F_2)$ we can obtain at most $2^{k-1}$ edges by applying a sequence of switching operations. (We use that the edge $(F_1,F_2)$ is the same as $(F_2,F_1)$.)

We say that an edge $(F_1, F_2)$ of $KG_{n, k}$ is \emph{interlacing} if the elements of $F_1$ and $F_2$ alternate along $[n]$. 
Finally, let $XG'_{n,k}$ be a graph with the vertex set ${[n]\choose k}$ and with the edge set that consists of  all edges $(F_1', F_2')$ which can be obtained from an interlacing edge via a sequence of switching operations.

In \cite{KS}, the authors showed that there is a subgraph $XG_{n,k}$ of $XG'_{n,k}$ with chromatic number $n-2k+2,$ which implies $\chi(XG'_{n,k})=n-2k+2.$

For a $2k$-element set $A$ there is one interlacing edge $(F_1, F_2)$ with $F_1\sqcup F_2 = A$. Thus, there are at most $2^{k-1}$ edges $(F_1, F_2)$ in $XG'_{n,k}$ with $F_1\sqcup F_2 = A$. On the other hand, $KG_{n,k}$ contains $\binom{2k-1}{k}$ edges with $F_1\sqcup F_2 = A$, which implies
$$
\frac{|E(KG_{n,k})|}{|E(XG'_{n,k})|} \ge 2^{-k+1}\binom{2k-1}{k} = 2^{-k}{2k\choose k}.
$$
Repeating the argument from Subsection~\ref{sec21}, if $KG_{n,k}$ has a subgraph $H$ with chromatic number $n-2k+2,$ then any coloring of $KG_{n, k}$ in $n - 2k + 1$ color has at least $\frac{|E(KG_{n,k})|}{|E(H)|}$ monochromatic edges. This and the inequality above imply the statement of the lemma.
\end{proof}

\subsubsection{Proof of Lemma~\ref{lem3_growk}}

For simplicity, suppose that $G$ is the induced subgraph of $KG_{n,k}(1/2)$ on $[h_0]$.
Here we apply the blow-up approach of the second author from \cite{Kup2} and amplify it using Lemma~\ref{lemkai}.

Fix $l = 6$ and consider a coloring of $G$ with $h_0 - 2k - 2l + 1$ colors. First, we construct a coloring of $KG_{h_0,k+l}$ with $h_0 - 2k - 2l + 1$ colors. We color each $(k+l)$-set in (one of) the most popular colors among its subsets. From Lemma~\ref{lemkai}, this coloring contains at least $2^{k-1} /k$ monochromatic edges and each of these edges corresponds to a monochromatic bipartite subgraph in $KG_{h_0,k}$ with parts of size at least $t := \binom{k + l}{k} / h_0$. Note that $t = Ck^{0.9}$ for some constant $C$.

Let us show that we can choose a large proportion of these $2^k / 2k$ monochromatic edges such that corresponding bipartite subgraphs are edge-disjoint. First, we need the following lemma.

\begin{lemma}
Let $n, r, s$ be positive integers such that $n > 2r^2 + 2r$ and $r>s$. Then there is a family $\mathcal H(n, r, s) \subset\binom{[n]}{r}$ of $r$-sets such that each $s$-set is contained in at most one set of $\mathcal H(n, r, s)$ and the number $dp(\mathcal H(n, r, s))$ of disjoint pairs of $r$-sets in $\mathcal H(n, r, s)$ is at least $\frac{1}{400} \binom{n}{r}{n-r\choose r} / \left(\binom{r}{s}\binom{n-r}{r-s}\right)^2$.
\end{lemma}

\begin{proof}
We apply the standard deletion method. Fix $p > 0$ and tаке a random subfamily $\mathcal H_1$ of $\binom{n}{r}$, including each $r$-set independently with probability $p$. The expected number of disjoint pairs in $\mathcal H_1$ is $\E [dp(\mathcal H_1)] = \frac12 p^2 \binom{n}{r}\binom{n-r}{r}$.
Then, for each pair of $r$-sets in $\mathcal H_1$ which intersect in a set of at least $s$ elements we remove one of those sets from $\mathcal H_1$, obtaining a family $\mathcal H$. By definition, each $s$-set is contained in at most one $r$-set from $\mathcal H$. It remains to bound from below the expected number of disjoint pairs in $\mathcal H$. We have that $dp(\mathcal H_1)-dp(\mathcal H)$ is at most  the number of triples $H_1, H_2, H_3\in\mathcal H_1$ such that $H_1\cap H_2 = \varnothing$, $H_1\ne H_3$, and $|H_1\cap H_3| \ge s$. Calculating the expectation, we get
$$
\E[dp(\mathcal H)] \ge \frac12 p^2 \binom{n}{r}\binom{n-r}{r} - p^3\binom{n}{r}\binom{n-r}{r} \sum_{i=s}^{r-1} \cdot \binom{r}{i}\binom{n - r}{r - i}.
$$

Taking $p^{-1} := 3\sum_{i=s}^{r-1} \binom{r}{i}\binom{n-r}{r-i}$, we obtain that there is a family $\mathcal H$ such that 
$$
dp(\mathcal H) \ge
\frac{1}{54} \binom{n}{r}\binom{n-r}{r} / \left(\sum_{i=s}^{r-1} \binom{r}{i}\binom{n-r}{r-i}\right)^2.
$$

For $n > 2r^2 + 2r$ and $i = s, \ldots, r-1$ we have 
$$
\binom{r}{i}\binom{n-r}{r-i} / \binom{r}{i+1}\binom{n-r}{r-i-1} = 
\frac{i+1}{r-i}\cdot \frac{n-2r+i+1}{r-i} > 2
$$
and $\sum_{i=s}^{r-1} \binom{r}{i}\binom{n-r}{r-i} \le 2 \binom{r}{s}\binom{n-r}{r-s}$. 
Thus,
\[
dp(\mathcal H) \ge 
\frac{1}{400}\binom{n}{r}{n-r\choose r} / \left(\binom{r}{s}\binom{n-r}{r-s}\right)^2.
\qedhere \popQED
\]
\end{proof}


Fix a family $\mathcal H := \mathcal H(h_0, k + l, k)$. Let $\pi$ be a random permutation of $[h_0]$, $\mathcal H_\pi$ be the corresponding image of $\mathcal H$, and $KG_{h_0,k+l}|_{\mathcal H_\pi}$ be the subgraph of $KG_{h_0,k+l}$ induced on the $(k+l)$-sets from $\mathcal H_\pi$. 
Clearly, a fixed pair of disjoint $(k+l)$-sets belongs to $\mathcal H_\pi$ with probability 
$$
dp(\mathcal H) \left(\frac12 \binom{h_0}{k+l}\binom{h_0 - k - l}{k+l} \right)^{-1} \ge \frac{1}{200}  \left( \binom{k+l}{k}\binom{h_0-k-l}{l} \right)^{-2}.
$$
Thus, there exists a permutation $\pi$ such that $KG_{h_0,k+l}|_{\mathcal H_\pi}$ contains at least $$\frac{2^k}{400 k} \left( \binom{k+l}{k}\binom{h_0-k-l}{l} \right)^{-2} \ge 2^k / (400k(k h_0)^{2l})$$ monochromatic edges.
As we noted before, each monochromatic edge of $KG_{h_0,k+l}|_{\mathcal H_\pi}$ corresponds to a monochromatic bipartite subgraph of $KG_{h_0,k}$ with parts of size at least $t$. Moreover,  due to the construction of $\mathcal H$, these bipartite subgraphs are edge-disjoint.

If the coloring is proper for the subgraph $G$, edges of these bipartite subgraphs must be missing in $KG_{n,k}(1/2)$. 
The expected number of such empty bipartite subgraphs in $KG_{h_0, k}$ is bounded from above by
$$
\binom{h_0}{k + l}^2 \binom{\binom{k+l}{k}}{t}^2 2^{-t^2} \le
2^{2(k + l) \log_2 h_0 + 2t l \log_2 (k + l) - t^2} =
2^{-(1 + o(1)) t^2}.
$$
(Here, we used that $t = Ck^{0.9}$ for some constant $C$.) Thus, the probability that there is a set $X$ of $h_0$ elements from $[n]$ such that there are $2^k / (400k(k h_0)^{2l})$ empty bipartite subgraphs as above in the subgraph $KG_{X,k}(1/2)$ of $KG_{n,k}(1/2)$ induced on $X$ is bounded by
$$
\binom{n}{h_0}  2^{-(1 + o(1)) t^2 \cdot  2^k / (400k(k h_0)^{2l})} \le
2^{k^{5.1} \log_2 n - 2^{(1+o(1))k}} = 
2^{k^{5.1} (\log_2 n - 2^{(1+o(1))k})} . 
$$
Since $k = (1 + \eps)\log_2 \log_2 n$, we have $2^{(1+o(1))k} = \log_2^{1+\eps+o(1)} n$, and the last displayed expression tends to 0 as $n\to\infty$. This concludes the proof of the lemma.

\section{Proof of Theorem~\ref{thm2}}\label{sec4}
The proof of the theorem follows a very natural approach: select a certain collection of edges in $KG_{n,k}$ and a subset of elements of $[n]$ that contains it, such that at least one copy of this collection is likely to disappear in the random graph. This, in turn, permits to color the corresponding subset of $[n]$ into relatively few colors. The other colors stay star-like.
The key idea in the proof of the bound, which is sharp at least in the case $r=2$, is the form of the collection of edges. \\

Consider an ordered family $\mathcal A=(A_1,\ldots, A_l)$ of $l$ pairwise disjoint $r$-element subsets of $[n]$.
For each $i\in[l]$, denote by $V_i(\mathcal A)$ the set of all $k$-element subsets of $A_1\sqcup\ldots\sqcup A_i$ containing at least one element from $A_i$ and denote by $E_i(\mathcal A)$ the set of all $r$-tuples of pairwise disjoint subsets in $V_i(\mathcal A)$. Note that any such $r$-tuple covers $A_i$.  Put $E(\mathcal A):=\bigcup_{i=1}^l E_i(\mathcal A)$.

Consider a spanning subgraph $G$ of $KG^r_{n,k}$.
We say that the family $\mathcal A$ is \emph{empty in $G$}, if, for each $i\in[l],$ the set $V_i(\mathcal A)$ is independent in  $G$ (in other words, $G$ does not contain edges from $E_i(\mathcal A)$).

\begin{lemma}\label{lemma21}
    If $G$ is a spanning subgraph of hypergraph $KG^r_{n,k}$, and there is a family $\mathcal A = \{A_1, \ldots, A_l\}$ as above that is empty in $G$, then $\chi(G) \le \left\lceil\frac{n-l}{r-1}\right\rceil$.
\end{lemma}
\begin{proof} We may  w.l.o.g. assume that $A_i = [r(i - 1) + 1, ri]$ for each $i\in[l]$. For a set $F$, denote $\max(F):=\max \{i: i\in F\}$.
Consider the following coloring $\kappa$ of the vertices of the graph $G$ in $\left\lceil\frac{n - l}{r-1}\right\rceil$ colors:
$$
\kappa(F) := \begin{cases}
    i, \ \ \ \ \ \ \ \ \ \ \ \ \ \ \ \ \ \ \ \mbox{ if } \max(F)\in A_i,\\
    l + \left\lceil \frac{\max(F) - rl}{r - 1} \right\rceil, \mbox{ if } \max(F) > rl.
\end{cases}
$$
Let us show that the coloring is indeed proper for $G$.
If pairwise disjoint sets $F_1, \ldots, F_r$ are colored in color $i \le l$, then they form an edge from $E_i(\mathcal A)$, and this edge does not belong to $G$ due to the fact that $\mathcal A$ is empty in $G$. Moreover, for any color $l+j$, $j>0$ the sets of that color all intersect the  segment $[rl+(r-1)j+1,rl+(r-1)(j+1)]$ of length $r-1$, and thus there are  no $r$ pairwise disjoint sets of that color. We conclude that the coloring $\kappa$ is proper for $G$ and that the total number of colors is
$\left\lceil\frac{n - l}{r-1}\right\rceil$.
\end{proof}

Combined with the previous lemma, the following lemma implies the statement of Theorem~\ref{thm2}.
\begin{lemma}\label{lemma22}
Let $r,k\ge 2$. There exists a constant $C=C(k,r)>0$, such that a.a.s. there is a family $\mathcal A=\{A_1, \ldots, A_{l}\}$ as above that is empty in $KG^r_{n,k}(1/2)$, where $l := C\sqrt[r(k-1)]{\log_2 n}$.
\end{lemma}

\begin{proof}
Note that the number of ordered families of $l$ pairwise disjoint $r$-sets is $N := \frac{n(n-1)\ldots(n-rl+1)}{(r!)^{l}}$ and the probability that such family is empty in $KG^r_{n,k}(1/2)$ is $\Prb := 2^{-|E(\mathcal A)|}$.
Order such families and denote by $X_i$ the indicator function of the event that the $i$-th family is empty in $KG^r_{n,k}(1/2)$. Let
$X := X_1+\ldots+X_N$ denote the number of empty families in $KG^r_{n,k}(1/2)$.

To prove the lemma, it is sufficient to show that $\frac{\D X}{(\E X)^2} \to 0$ as $n\to\infty$. Indeed, from Chebyshev's inequality
$$
\Pr(X = 0) \le
\Pr(|X - \E X| \ge \E X) \le
\frac{\D X}{(\E X)^2}.
$$

We rewrite the expression  $\frac{\D X}{(\E X)^2}$ in the following way:
$$
\frac{\D X}{(\E X)^2} =
\frac{\sum_{i=0}^N \sum_{j=0}^N (\E X_i X_j - \E X_i \E X_j)}{N^2 \Prb^2} =
\frac1{N^2} \sum_{i=1}^N \sum_{j=0}^N \frac{\E X_iX_j - \Prb^2}{\Prb^2}.
$$

Let us estimate $\frac{\E X_iX_j - P^2}{P^2}$. Let $\mathcal A:=\{A_1, \ldots, A_{l}\}$  and $\mathcal B:=\{B_1, \ldots, B_{l}\}$ be the $i$-th and the $j$-th family, respectively. Then $\E X_iX_j = 2^{- |E(\mathcal A)\cup E(\mathcal B)|}$. Recall that $\Prb = 2^{-|E(\mathcal A) |} = 2^{-|E(\mathcal B)|}$. Then
$$
\frac{\E X_iX_j - \Prb^2}{\Prb^2}
=
\frac{2^{- |E(\mathcal A) \cup E(\mathcal B)|}}{2^{-|E(\mathcal A)|} 2^{-|E(\mathcal B)|}} - 1
=
2^{|E(\mathcal A) \cap E(\mathcal B)|} - 1.
$$

Let us estimate $|E(\mathcal A) \cap E(\mathcal B)|$. Any edge from the intersection is an $r$-tuple of $k$-sets, each of which is a subset of the set $C := (A_1\sqcup\ldots\sqcup A_{l}) \cap (B_1\sqcup\ldots\sqcup B_{l})$. Therefore,  if for some $s\in [l]$ we have $A_s \not\subset C$, then $|E_s(\mathcal A) \cap E(\mathcal B)| = 0$. Otherwise,
$|E_s(\mathcal A) \cap E(\mathcal B)| \le \binom{|C|}{k-1}^r$.
Then
$$
|E(\mathcal A) \cap E(\mathcal B)| \le
\frac{|C|}{r} \binom{|C|}{k-1}^r \le
|C|^{r(k-1)+1}.
$$

Therefore
$$
\frac{\E X_iX_j - \Prb^2}{\Prb^2}
\le
2^{|C|^{r(k-1)+1}} - 1.
$$

Note that, for a given family $\mathcal A$, the number $N_m$ of families $\mathcal B$ such that $|(A_1\sqcup\ldots\sqcup A_{l}) \cap (B_1\sqcup\ldots\sqcup B_{l})| = m$ satisfies $N_m\le (rl)^{2m} n^{rl - m}=n^{rl-(1-o(1))m}$. Due to the condition on $l$, one can choose $C>0$ in the formulation of the lemma such that  for each $m\in[rl]$ we have
\begin{equation}\label{eqhelp} 
2^{m^{r(k-1)+1}} = 
o(2^{\frac{m}3\log_2n}) =
o(n^{m/3}).
\end{equation}

We are ready to conclude the proof.
\begin{align*}
\frac{\D X}{(\E X)^2} =&\
\frac1{N^2} \sum_{i=1}^N\sum_{j=0}^N \frac{\E X_iX_j - \Prb^2}{\Prb^2} \le
\frac1{N^2} \sum_{i=1}^N \sum_{m=1}^{rl} N_m \left(2^{m^{r(k-1)+1}} - 1\right) \\
=&\ \frac1{N} \sum_{m=1}^{rl} N_m \left(2^{m^{r(k-1)+1}} - 1\right) \le
 \frac1N \sum_{m=1}^{rl} n^{rl - (1+o(1))m} 2^{m^{r(k-1)+1}}\\ 
 \overset{\eqref{eqhelp}}{\le}&\
\frac 1N\sum_{m=1}^{rl}  n^{rl - m/3}
\le (1+o(1))\frac{n^{rl-1/4}(r!)^l}{n^{rl}}\to 0\ \ \text{as} \ \ n\to \infty,
\end{align*}
because for some $c>0$ we have $(r!)^l =O\Big(c^{\sqrt[r(k-1)]{\log_2 n}}\Big) =  O\Big(c^{\sqrt{\log_2 n}}\Big) = o(n^{1/4})$.
\end{proof}

\subsection{Proof of the upper bound from Theorem~\ref{thm1} for $k=2$}\label{sec41}

The  proof is similar to the proof of Theorem~\ref{thm2}, but instead of selecting an empty family in one turn, we select independent sets one after another.

Put $h_1 := \sqrt[2]{\frac16 \log_2 n \cdot\log_2 \log_2 n}$, $A :=  (\log_2 n)^{5/6}$ and $N_i :=  \frac{\log_2 n}{2(h_1 + i - 1)}$. Note that
$$
\sum_{i=1}^A N_i \ge
\sum_{i=1}^A \frac{\log_2 n}{2(h_1 + i - 1)} \ge
\frac{\ln 2}{2} \log_2 n \cdot \log_2 \frac{h_1 + A}{h_1}
$$
\begin{equation}
\label{k2_eq_sum}
\ge
(1 + o(1)) 0.11 \cdot \log_2 n \cdot \log_2 \log_2 n >
1.1\binom{h_1}{2}.
\end{equation}

Fix an ordering of $[n]$ and consider $[h_1]$. Using \eqref{k2_eq_sum} and the fact that $N_i = o(h_1)$, it is easy to see that there is a partition $\mathcal A_1, \ldots, \mathcal A_A$ of the family $\binom{[h_1]}{2}$ such that $|\mathcal A_i| \le N_i$ and each family $\mathcal A_i$ is intersecting.
We aim to reorder the elements of $[h_1 + 1, n]$ so that, for each $i\in[A]$, the family $\mathcal A_i \cup \left( \binom{[h_1 + i]}{2} \cap \mathcal S_{h_1+i} \right)$ is independent in $KG_{n,k}(1/2)$. This obviously implies that there is a coloring of $KG_{n,2}(1/2)$ in $n-h_1$ colors.

Assume that we have found suitable elements $h_1 + 1, \ldots, h_1 + i - 1$ and let us estimate the probability $\Prb_i$ that we can choose the element $h_1 + i$. The family $\mathcal A_i \cup \left( \binom{[h_1 + i]}{2} \cap \mathcal S_{h_1+i} \right)$ induces at most $(h_1 + i - 1) N_i$ edges in $KG_{n,2}$, and therefore
$$
\Prb_i \ge
1 - (1 - 2^{-(h_1+i - 1) N_i})^{n - h_1 - i + 1} =
1 - (1 - 2^{-\frac{(1 + o(1))\log_2 n}{2}})^{n - h_1 - i + 1}
$$
$$
\ge
1 - (1 - n^{-2/3})^{n - h_1 - i + 1} \ge
1 - e^{-\frac12 n^{1/3}}.
$$

Then the probability $\Prb$ that we can assign all elements $[h_1 + 1, h_1 + A]$ is at least
$$
\Prb \ge 
1 - \sum_{i=1}^A(1 - \Prb_i) \ge
1- A e^{-\frac12 n^{1/3}} \ge 1-e^{-n^{1/4}}
\to 1.
$$

\section{Conclusion}\label{sec5}
For clarity, all our results are stated and proved for $p=1/2$. With the same analysis, we could extend them to a bigger range of values of $p$. However, we are missing the understanding of the behaviour of $\chi(KG_{n,k}(p))$ in sparser regimes. In particular, what is the threshold for $\chi(KG_{n,k}(p))\ge n/2$?


Thanks to the paper of Kaiser and Stehl\'{\i}k \cite{KS}, we have exponential upper and lower bounds on $\zeta$ from Problem~\ref{pr2}: roughly speaking, $2^k\le \zeta \le 4^k.$ If $\zeta  = 2^{\alpha k}$ for some $1\le \alpha\le 2$, then the methods of this paper and the paper \cite{Kup2} of the second author would allow to show that the correct bound for $k$ for which Theorem~\ref{thm3} holds is $k =(\alpha^{-1}+o(1))\log_2\log_2 n$.  Unfortunately, the result from \cite{KS} does not give anything for the  Conjectures~\ref{conj1} and~\ref{conj2}, and so they remain wide open.

Another natural question is to obtain a structural result for Problem~\ref{pr66} that would work for $n\ll k^2$: i.e., to prove that, depending on $n = n(k)$, some of the colors in the proper colorings into $n-2k+2$ colors must be close (in some of the possible senses) to a star. 

Finally, in the hypergraph case, it would be interesting to close the gap between the lower bound \eqref{eq03} from \cite{Kup2} and the upper bound from Theorem~\ref{thm2}.
\\

{\sc Acknowledgements. } We thank Florian Frick and G\'abor Tardos for useful discussions on Conjecture~\ref{conj2}. Florian pointed out the connection to Sarkaria's inequality.

\begin{small}

\end{small}

\end{document}